\documentclass[11pt]{amsart}
\usepackage[utf8]{inputenc}
\usepackage[T1]{fontenc}
\usepackage{fixltx2e}
\usepackage{graphicx}
\usepackage{longtable}
\usepackage{float}
\usepackage{wrapfig}
\usepackage{soul}
\usepackage{textcomp}
\usepackage{marvosym}
\usepackage[nointegrals]{wasysym}
\usepackage{latexsym}
\usepackage{amssymb}
\usepackage{hyperref}
\usepackage[all]{pabmacros}
\usepackage{amsmath}
\usepackage[bibstyle=alphabetic,citestyle=alphabetic]{biblatex}
\bibliography{refs}
\AtEveryBibitem{\clearfield{doi}}
\AtEveryBibitem{\clearfield{url}}
\AtEveryBibitem{\clearfield{issn}}
\renewcommand*{\bibfont}{\footnotesize}
\DeclareMathOperator{\harmean}{H}
\DeclareMathOperator{\isoconst}{\mathcal{I}}
\renewcommand{\genus}{\operatorname{\lambda}}
\usepackage{snapshot}
\address{Department of Mathematics, University of California San Diego}
\email{pbryan@ucsd.edu}
\keywords{Ricci flow, Isoperimetric profile, curvature}
\subjclass[2010]{53C44, 35K55, 58J35}
\date{}

\title{Curvature bounds via an isoperimetric comparison for Ricci flow on surfaces}
\author{Paul Bryan}

\hypersetup{
  pdfkeywords={},
  pdfsubject={},
  pdfcreator={Emacs Org-mode version 7.8.02}}

\begin{document}

\maketitle

\begin{abstract}
A comparison theorem for the isoperimetric profile on the universal cover of surfaces evolving by normalised Ricci flow is proven. For any initial metric, a model comparison is constructed that initially lies below the profile of the initial metric and which converges to the profile of the constant curvature metric. The comparison theorem implies that the evolving metric is bounded below by the model comparison for all time and hence converges to the constant curvature profile. This yields a curvature bound and a bound on the isoperimetric constant, leading to a direct proof that the metric converges to the constant curvature metric.
\end{abstract}

\section{Introduction}
\label{sec-1}

The Ricci flow is the nonlinear geometric parabolic evolution equation
\begin{equation}
\label{eq:RF}
\begin{cases}
\pd{t} \metric &= -2\ricci(t) \\
\metric(0) &= \metric_0
\end{cases}
\end{equation}
for a smooth family of Riemannian metrics $\metric(t)$ on a smooth manifold $M$ with Ricci curvature $\ricci(t)$ and an arbitrary smooth initial metric $\metric_0$. Here we are interested in the case of closed surfaces, that is, $2$ dimensional, compact manifolds $M$ without boundary. The results here pertain to the \emph{normalised} flow, preserving the $2$-dimensional volume of $M$. After rescaling the initial metric to have volume $4\pi$ and applying the Gauss-Bonnet formula, the normalised flow on surfaces takes the form
\begin{equation}
\label{eq:NRF}
\pd{t}\metric = -2(\gausscurv - \avg{\gausscurv})\metric
\end{equation}
where $\gausscurv$ is the Gaussian curvature and $\avg{\gausscurv} = \frac{1}{4\pi} \int_M \gausscurv \metricmeasure{\metric}$ is the average Gauss curvature on $M$ \cite{MR2729306}. An important consequence of writing the equation in this way is that it may be lifted to the universal cover $\unicover{M} \to M$. That is, the pullback metric $\unicover{\metric}(t) = \pullback{\uniproj} \metric (t)$ also evolves according to equation \eqref{eq:NRF}.

The main theorem of this paper is a comparison theorem for the isoperimetric profile of a surface with metric evolving by the normalised Ricci flow, generalising the comparison theory in \cite{MR2729306} for $M=\sphere^2$ to arbitrary closed surfaces. Recall, the isoperimetric profile is the least boundary area enclosing a given volume (see section \ref{sec:isoprofile} for a precise definition). 
\begin{theorem}
[Main Theorem \ref{thm:comparison}]
Let $(M, \metric(t))$ be a Ricci flow of a closed surface and $(\unicover{M}, \unicover{\metric}(t))$ be the lift to the universal cover.  Let $\phi: (0, \abs{\unicover{M}}) \times [0, T) \to \RR$ be a smooth, strictly positive, strictly concave function satisfying
\[
\pd[\phi]{t} \leq \phi'' \phi^2 - (\phi')^2 \phi + \phi'\left(4\pi - 2(1-\genus)a_0\right) + (1-\genus)\phi.
\]
along with the asymptotic behaviour
\[
\limsup_{a\to 0} \frac{\phi(a, t)}{\sqrt{4\pi a}} \leq  1
\]
and
\[
\limsup_{a\to \infty} \left(\isoprofile(a, t) - \phi(a,t)\right) \geq 0.
\]

If the initial inequality, $\phi(a, 0) < \isoprofile_{\unicover{\metric}(0)} (a)$ for all $a \in (0, \abs{\unicover{M}})$ holds, then $\phi(a, t) \leq \isoprofile_{\unicover{\metric}(t)} (a)$ for all $a, t$ with strict inequality if the inequality in \eqref{eq:ricci_iso_diff_inequal} is strict.
\end{theorem}

As an application, by a standard bootstrapping argument, a proof of the Hamilton-Chow theorem (Theorem \ref{thm:convergence}) is obtained directly as a corollary of Theorem \ref{thm:comparison}. This achieved by suitable choices of comparison functions, leading to explicit curvature bounds and bounds on the isoperimetric constant as described in sections \ref{sec:models} and \ref{sec:convergence}.

\begin{theorem}
[\cite{MR954419, MR1094458}]
\label{thm:convergence}
Given any initial metric $\metric_0$, there exists a unique solution to the normalised Ricci flow existing for all time $t \in [0,\infty)$ and such that $\metric(t) \to_{C^{\infty}} \metric_{\genus}$, the metric of constant curvature $\gausscurv = 1-\genus$ as $t\to\infty$.
\end{theorem}

The use of the isoperimetric profile here is an extension from the 2-sphere to arbitrary surfaces of the results in \cite{MR2729306}, which in turn are based on the isoperimetric estimates in \cite{MR1369139}. We begin in section \ref{sec:isoprofile} with a treatment of the isoperimetric profile of surfaces, deriving a viscosity equation via variational techniques which forms the heart of the comparison theorem. The comparison theorem is proven in section \ref{sec:comparison} by coupling the time-variation of the isoperimetric profile under the normalised Ricci flow with the spatial viscosity equation, yielding the parabolic version of the viscosity equation. Section \ref{sec:models} is devoted to constructing suitable model comparisons. The construction on the $2$-sphere was given in \cite{MR2729306} and is briefly described. Curiously, the most difficult case to deal with is for surfaces of genus $\genus > 1$, which historically was perhaps the easiest case by applying the maximum principle and introducing a potential function \cite{MR954419}. For initially negatively curved surfaces however, the model given here is quite appealing. Finally in section \ref{sec:convergence}, the boostrapping convergence argument is briefly described.

\section*{Acknowledgements}
I would like to thank Professor Ben Andrews for his generous guidance and support supervising my Ph.D. thesis on which this paper is based. I would also like to thank Professor Gang Tian and BICMR for sponsoring a very enjoyable stay in Beijing where my thoughts on variational techniques were greatly clarified on the banks of WeiMing Lake. Finally, this paper was written during my time as SEW Assistant Professor at UCSD.
\section{The Isoperimetric Profile}
\label{sec-2}
\label{0b158ef8-2a10-4c0e-9d45-0de44bbddf06}

\label{sec:isoprofile}
\subsection{Definition and Basic Properties}
\label{sec-2-1}

\begin{defn}
The \emph{isoperimetric profile}, $\isoprofile_M : (0, \abs{M}) \to \RR_+$ of $M$ is defined by
\[
\isoprofile_M(a) = \inf \left\{\abs{\bdry{\Omega}} : \abs{\Omega} = a \right\}
\]
where the infimum is taken over all relatively compact open sets $\Omega$ with smooth boundary.  Such $\Omega$ are said to be \emph{admissible regions}.  If $\Omega$ is an admissible region such that $\isoprofile_M(\abs{\Omega}) = \abs{\bdry{\Omega}}$, we will call $\Omega$ an \emph{isoperimetric region}.
\end{defn}

A basic theorem we will assume here is that for every $a\in(0,\abs{M})$, there exists a corresponding isoperimetric region (with smooth boundary apart from a set of Hausdorff dimension at most $n-7$ on an $n$ dimensional manifold) provided $M$ is either compact or co-compact. In particular, smooth isoperimetric regions exist on a closed surface and its universal cover equipped with the pull-back metric. The proof of this fact is a standard result of geometric measure theory \cite[pp. 128-129]{MR2455580}. A simplified proof in the case of surfaces is given in \cite{MR1661278} using regularity techniques developed in \cite{MR1417428}.

It will be important for us to understand the behaviour of the isoperimetric profile near the end points $\{0, \abs{M}\}$.  In the situation where $M$ is compact, then the complement of an isoperimetric region is again an isoperimetric region, so the isoperimetric profile is symmetric about $\abs{M}/2$ and it suffices to consider only the behaviour near $0$.  In the non-compact case, the behaviour near $0$ is the same as for the compact case, so let us begin with the behaviour near $0$.

\begin{theorem}
\label{thm:bdrybehaviour}
Let $M$ be a smooth Riemannian surface without boundary and such that $\sup_M \gausscurv < \infty$.  Then the isoperimetric profile satisfies
\[
\isoprofile(a)= \sqrt{4\pi a} - \frac{\sup_M \gausscurv}{4\sqrt{\pi}} a^{3/2} + O(a^{5/2}) \quad \text{as} \quad a\to 0.
\]
\end{theorem}

\begin{proof}
Small geodesic balls about any point $p$ are admissible regions.  The result of \cite[Theorem 3.1]{MR0339002} gives $\abs{B_r(p)} = \pi r^2 \left(1 - \frac{\gausscurv(p)}{12} r^2 + O(r^4)\right)$ and $\abs{\bdry{B_r(p)}} = 2\pi r \left(1 - \frac{\gausscurv(p)}{6} r^2 +O(r^4)\right)$.  The upper bound follows since $\abs{\bdry{B_r(p)}} \geq \isoprofile(\abs{B_r(p)})$.

To prove the lower bound, first choose $a_0$ sufficiently small to ensure that $\isoprofile(a_0)$ is much smaller than the injectivity radius of $M$.  Then an isoperimetric region $\Omega_0$ corresponding to $a_0$ lies inside a geodesic ball about some point $p$ (width is bounded above by perimeter for surfaces).  Since geodesic balls are simply connected and $\gausscurv \leq \gausscurv_0 = \sup_M\gausscurv$, the Bol-Fiala inequality (see \cite{MR0500557}) then gives
\begin{align*}
\isoprofile(a_0) \geq \sqrt{4\pi a_0 - \gausscurv_0 a_0^2} = \sqrt{4\pi a_0} - \frac{\gausscurv_0}{4\sqrt{\pi}} a_0^{3/2} + O(a_0^2).
\end{align*}
\end{proof}

Next, we have the asymptotics of the isoperimetric profile near $\infty$ for non-compact $\unicover{M}$.

\begin{theorem}
\label{thm:asymptotics_large}
Let $M$ be a closed, genus $\genus\geq 1$ surface with metric $\metric$, normalised to have $\abs{M} = 4\pi$ and let $\uniproj: \unicover{M} \to M$ be the universal cover of $M$ equipped with the pull-back metric $\unicover{\metric} = \uniproj^{\star}\metric$.  Then
\[
\isoprofile_{\unicover{\metric}}(a) \to C \sqrt{4\pi a - (1-\genus) a^2}
\]
as $a\to\infty$ for some $C>0$.
\end{theorem}

\begin{proof}
By the uniformisation theorem, $\unicover{\metric}$ is conformal to a metric of constant curvature so that
\[
\unicover{\metric} = \phi\metric_{1-\genus}
\]
with $\metric_{1-\genus}$ the metric of constant curvature $1-\genus$ and $\phi$ a positive function $\phi: \unicover{M} \to \RR$ invariant under the deck transformation group of $\unicover{M}$.  Thus $\phi$ is uniformly bounded above and below.

The isoperimetric inequality for simply connected Riemannian surfaces of constant curvature $1-\genus$ implies that the isoperimetric profile $I_{1-\genus}$ of the constant curvature metric $\metric_{1-\genus}$ is given by
\[
\isoprofile_{1-\genus}(a) = \sqrt{4\pi a - (1-\genus)a^2}.
\]

Since $\unicover{\metric}$ is conformal to $\metric_{1-\genus}$ with conformal factor $\phi$ uniformly bounded, we have
\begin{align*}
\frac{1}{C_1} \abs{\bdry{\Omega}}_{\metric_{1-\genus}} &\leq \abs{\bdry{\Omega}}_{\unicover{\metric}} \leq C_1 \abs{\bdry{\Omega}}_{\metric_{1-\genus}} \\
\frac{1}{C_2} \abs{\Omega}_{\metric_{1-\genus}} &\leq \abs{\Omega}_{\unicover{\metric}} \leq C_2 \abs{\Omega}_{\metric_{1-\genus}}
\end{align*}
for constants $C_1,C_2>0$ which gives the result.
\end{proof}

\begin{remark}
It would be preferable if we didn't have to refer to the uniformisation theorem, as then the results of this paper provide a proof of the uniformisation theorem.  In the case $\genus=0$, we have such a proof since $M$ is compact.  In the case $\genus=1$, the result of \cite{MR1354290} implies that
\[
\isoprofile(a) \to C\sqrt{a}
\]
as $a\to \infty$ for $0 < C \leq 4\pi$ with $C=4\pi$ if and only if $M$ is flat.  This is precisely the required asymptotics in the theorem for $\genus=1$ surfaces obtained without requiring the use of the uniformisation theorem.  The only problem here then is for $\genus>1$ surfaces. The volume growth of $\unicover{M}$ is controlled by the number of generators for the fundamental group, but controlling the perimeter is rather more difficult. I am not aware of an applicable result for $\genus>1$ surfaces, though such a result would be interesting.
\end{remark}
\subsection{Variational Formulae and Consequences}
\label{sec-2-2}

Our techniques are based on applying the standard variational formula for isoperimetric regions, with a slight change in the second variation, obtained by applying the Gauss-Bonnet theorem.  Let us briefly recall the applicable variational formulae and describe the approach used here in obtaining the second variation.

Let $\Omega_0$ be an isoperimetric region and $\Omega_{\epsilon}$ a smooth normal variation with variational vector field $\eta \nor$ for a smooth function $\eta : \bdry{\Omega_0} \to \RR$. The first variation formulae are
\begin{align}
\label{eq:firstvar_bdry}
\pd{\epsilon} \abs{\bdry{\Omega_{\epsilon}}} &= \int_{\bdry{\Omega_{\epsilon}}} \eta \curvecurv \\
\intertext{and}
\label{eq:firstvar_vol}
\pd{\epsilon} \abs{\Omega_{\epsilon}} &= \int_{\bdry{\Omega_{\epsilon}}} \eta .
\end{align}
where $\curvecurv$ is the geodesic curvature of $\bdry{\Omega_{\epsilon}}$. In particular, the vanishing of the first variation for all functions $\eta$ such that $\int_{\bdry{\Omega_{\epsilon}}} \eta = 0$ (area preserving variations) implies that $\curvecurv$ is constant.

For the second variation, we need only consider unit-speed variations ($\eta \equiv 1$) and so immediately conclude the second variation for area,
\begin{equation}
\label{eq:secondvar_vol}
\pdd{\epsilon} \abs{\Omega_{\epsilon}} = \int_{\bdry{\Omega_{\epsilon}}} \curvecurv.
\end{equation}

For the second variation of boundary length, it suits our purposes to first apply the Gauss-Bonnet formula and then differentiate equation \eqref{eq:firstvar_bdry}.  Thus
\[
\begin{split}
\pdd{\epsilon} \abs{\bdry{\Omega_{\epsilon}}} &= \pd{\epsilon} \int_{\bdry{\Omega_{\epsilon}}} \curvecurv \\
&= \pd{\epsilon} \left(2\pi\eulerchar{\Omega(\epsilon)} - \int_{\Omega(\epsilon)}\gausscurv_M\right)
\end{split}
\]
where $\eulerchar{\Omega{\epsilon}}$ is the Euler characteristic of $\Omega_{\epsilon}$ which is independent of $\epsilon$ since each $\phi_{\epsilon}$ is a diffeomorphism and $\gausscurv_M$ is the Gauss curvature of $M$.  The latter has no explicit dependence on $\epsilon$ and so the Reynold's Transport Theorem (or differentiating under the integral sign) yields
\begin{equation}
\label{eq:secondvar_bdry}
\pdd{\epsilon} \abs{\bdry{\Omega_{\epsilon}}} = - \int_{\bdry{\Omega_{\epsilon}}} \gausscurv_M.
\end{equation}

Our approach is based on weak differential inequalities for the isoperimetric profile arising from the variational formulae. 

\begin{defn}
A function $f: (a,b) \to \RR$ has weak derivatives satisfying
\[
\pd[f^-]{x} \leq C_1 \leq \pd[f^+]{x} \quad \text{and} \quad \pdd[f]{x} \leq C_2
\]
in the \emph{support} (or sometimes \emph{Calabi}) sense at $x_0$ if $f$ supports a smooth function $\phi$ at $x_0$ ($f(x_0)=\phi(x_0)$ and $f(x) \leq x_0$ for $x$ near $x_0$) such that
\[
\pd[\phi]{x} (x_0) = C_1 \quad \text{and} \quad \pdd[\phi]{x} (x_0) = C_2.
\]
\end{defn}

\begin{prop}
[\cite{MR875084} (see also \cite{MR2229062} pp. 249-251)]
\label{prop:support_iso_nobdry}
For each $a_0 \in (0,\abs{M})$, let $\Omega_0$ be a corresponding isoperimetric region with constant curvature $\curvecurv(a_0)$ along the boundary.  Then the isoperimetric profile satisfies
\[
\pd[\isoprofile^-]{a} \leq \curvecurv(a_0) \leq \pd[\isoprofile^+]{a} \quad \text{and} \quad \pdd[\isoprofile]{a} \leq \frac{-1}{\isoprofile^2}\left(\curvecurv(a_0)^2\isoprofile + \int_{\bdry{\Omega_0}} \gausscurv_M\right).
\]
in the support sense.  Moreover, if $\gausscurv_M \geq \gausscurv_0$, the function
\[
a \mapsto \isoprofile(a)^2 + \gausscurv_0 a^2
\]
is concave, hence $\isoprofile^2$ is locally Lipschitz and in particular $\isoprofile$ is continuous.
\end{prop}

\begin{remark}
Since we are assuming $M$ is compact or co-compact, $\gausscurv_M$ is bounded hence $\isoprofile$ is continuous.  Note also that since $\sqrt{-}$ is smooth away from $0$, by Rademacher's theorem, $\isoprofile$ is differentiable almost everywhere.
\end{remark}

\begin{cor}
\label{cor:iso_concave_nobdry}
With the notation of the proposition, if $\gausscurv_0\geq0$ then $\isoprofile$ is concave and so too is $\isoprofile^2$.  If the inequality is strict, then $\isoprofile$ and $\isoprofile^2$ are strictly concave.
\end{cor}

The last results of this section concern the topology of isoperimetric regions.  We generally don't have a priori control over the topology of isoperimetric regions and so we don't know the precise form of the differential inequality for $\isoprofile$ because of the integral over the unknown regions $\Omega_0$.  However, there is a useful sufficient condition for obtaining control of the topology of isoperimetric regions.  The idea comes from \cite{MR1674097}.

\begin{lemma}
\label{thm:curvetopology}
Let $a_0\in (0,\abs{M})$ and $\Omega_0$ a corresponding isoperimetric region.  If there exists a strictly positive, strictly concave function $\phi: (0,\abs{M}) \to \RR$ supporting $I$ at $a_0$ ($\phi(a_0) = \isoprofile(a_0)$ and $\phi(a) \leq \isoprofile(a)$ for all $a\in(0,\abs{M})$), then $\Omega_0$ is connected.  If $M$ is compact then $\Omega_0$ has connected complement.
\end{lemma}

\begin{remark}
It is worth pointing out that while the conclusion of the lemma is local, pertaining to a particular value of $a_0$ and corresponding isoperimetric region, the hypotheses are global in nature in that we need a \emph{globally defined} supporting function $\phi$ (not just in a neighbourhood of $a_0$).
\end{remark}

\begin{proof}
First note that since $\phi > 0$ on $(0,\abs{M})$, $\phi \leq \isoprofile$, and $\isoprofile(0) = 0$, we have $\phi(0) = 0$. Thus since $\phi$ is strictly concave, $\phi$ is strictly subadditive.

Now suppose $\Omega_0$ is not connected.  Then we can write $\Omega_0 = \Omega_1\cup\Omega_2$ with $\Omega_1 \cap \Omega_2 = \emptyset$.  Since $\bdry{\Omega_0}$ is smooth we must have $\bdry{\Omega_0} = \bdry{\Omega_1} \cup \bdry{\Omega_2}$ and $\bdry{\Omega_1} \cap \bdry{\Omega_2} = \emptyset$.  Thus we have $\abs{\Omega_0} = \abs{\Omega_1} + \abs{\Omega_2}$ and $\abs{\bdry{\Omega_0}} = \abs{\bdry{\Omega_1}} + \abs{\bdry{\Omega_2}}$, and all of these are non-zero.  But then we get
\begin{align*}
\phi\left(\abs{\Omega_1}\right) + \phi\left(\abs{\Omega_2}\right) &\leq \abs{\bdry{\Omega_1}} + \abs{\bdry{\Omega_2}} \\
&= \abs{\bdry{\Omega_0}} \\
&= \phi\left(\abs{\Omega_0}\right) \\
&= \phi\left(\abs{\Omega_1} + \abs{\Omega_2}\right) \\
&< \phi\left(\abs{\Omega_1}\right) + \phi\left(\abs{\Omega_2}\right).
\end{align*}
This is a contradiction, so $\Omega_0$ is connected.

If $M$ is compact, then $M\setminus\Omega_0$ is also an isoperimetric region with $\abs{M\setminus\Omega_0} = \abs{M}-\abs{\Omega_0} = \abs{M} - a_0$ and $\isoprofile(\abs{M} - a_0) = \abs{\bdry{M\setminus\Omega_0}} = \abs{\bdry{\Omega_0}} = \isoprofile(a_0)$.  Reflecting $\phi$ about $a=\abs{M}/2$ gives a function satisfying the hypothesis of the proposition at $\abs{M}-a_0$ hence $M\setminus\Omega_0$ is also connected.
\end{proof}

\begin{cor}
\label{cor:sphere_iso_conn}
With the hypothesis of lemma \ref{thm:curvetopology}, if $M$ is diffeomorphic to $\sphere^2$ then $\Omega_0$ is simply connected.
\end{cor}

\begin{proof}
Follows from the Jordan curve theorem for $\sphere^2$.
\end{proof}

\begin{cor}
\label{cor:plane_iso_conn}
With the hypothesis of lemma \ref{thm:curvetopology}, if $M$ is diffeomorphic to $=\RR^2$ then $\Omega_0$ is simply connected.
\end{cor}

\begin{proof}
Since $\RR^2$ is not compact, we cannot immediately conclude that $\Omega_0$ has connected complement as before. To achieve this result, first note that if $M$ is $\RR^2$, then $\phi$ is a (strictly) positive concave function on $(0,\infty)$ and hence is strictly increasing.  Since $\Omega_0$ is connected, topologically it is a disc with finitely many discs removed.  Let $\Omega_1$ denote the interior of the external boundary of $\Omega_0$, i.e. $\Omega_1$ is equal to $\Omega_0$ with the ``holes'' filled in.  Then $\Omega_1$ has strictly larger area than $\Omega_0$ and strictly smaller boundary length.  But then
\[
\phi(\abs{\Omega_0}) = \isoprofile(\abs{\Omega_0}) = \abs{\bdry{\Omega_0}} > \abs{\bdry{\Omega_1}} \geq \isoprofile(\abs{\Omega_1}) \geq \phi(\abs{\Omega_1})
\]
contradicting that $\phi$ is increasing.  Therefore $\RR^2 \setdiff \Omega_0$ is connected and now the Jordan curve theorem implies $\Omega_0$ is simply connected.
\end{proof}

Let us finish by noting that in positive curvature, we have complete knowledge of the topology of isoperimetric regions.

\begin{cor}
\label{cor:isoregions_positive_curvature}
If $M$ is diffeomorphic to either $\sphere^2$ or $\RR^2$ (for instance if $M$ is the universal cover of a closed surface) equipped with any metric (not necessarily the pull-back from a compact surface) and $\gausscurv_0 > 0$, then all isoperimetric regions are simply connected.
\end{cor}

\begin{proof}
By Corollary \ref{cor:iso_concave_nobdry}, $\isoprofile$ is strictly concave so the hypotheses of Corollaries \ref{cor:sphere_iso_conn} and \ref{cor:plane_iso_conn} are satisfied at any $a_0\in (0,\abs{M})$ by choosing $\phi=\isoprofile$ itself.
\end{proof}

\begin{remark}
I don't know if this result can be extended to $\gausscurv_0 = 0$ since in this case $\isoprofile$ is not necessarily strictly concave.
\end{remark}
\subsection{A viscosity equation for the isoperimetric profile}
\label{sec-2-3}

The results in this section formalise some of the ideas used in \cite{MR2729306}.  We obtain a differential inequality in the viscosity sense for the isoperimetric profile of a surface.  This is somewhat dual to the results in the previous section and those in \cite{MR2041647} in that we assert conditions which \emph{lower} supporting functions must satisfy as opposed to the aforementioned results which assert the existence of an \emph{upper} supporting function with bounds on the derivatives.  The methods however, are essentially the same and the support inequality implies the viscosity inequality.  As the isoperimetric profile is defined as an extrema, viscosity equations turn out to be well suited to this situation.  Indeed, viscosity equations were introduced in \cite{MR690039} to study Hamilton-Jacobi equations, also arising from optimisation problems.  A central feature of viscosity equations, that forms the basis of the comparison theorem \ref{thm:comparison}, is that they enjoy a maximum principle. See \cite{MR1351007} for more on viscosity equations.

\begin{defn}
A lower semi-continuous function $f: (a,b) \to \RR$ is a viscosity super-solution of the 2nd order differential equation
\[
A(x, f, f', f'') = 0
\]
if for every $x_0\in (a,b)$ and every $C^2$ function $\phi$ such that $\phi(x_0) = f(x_0)$ and $\phi(x) \leq f(x)$ in a neighbourhood of $x_0$, we have $A(x_0, \phi(x_0), \phi'(x_0), \phi''(x_0)) \geq 0$.  An upper semi-continuous function is a viscosity sub-solution if the same statements hold with all the inequalities reversed.
\end{defn}

For $f$ a viscosity super(sub)-solution of $A(x, f, f', f'') = 0$, we will abuse notation slightly and write $A(x, f, f', f'') \geq 0 (\leq 0)$ (in the viscosity sense).  

\begin{remark}
In the definition, the existence of a lower (upper) supporting function at a point is not required, rather we assert that if such a supporting function exists, it must satisfy the appropriate differential inequality.  For instance, the absolute value function, $x\mapsto \abs{x}$ is a viscosity sub-solution of $f'=1$ even though no $C^1$ upper support function exists at $x=0$.
\end{remark}

\begin{theorem}
\label{thm:viscosity_nobdry}
The isoperimetric profile is a viscosity super-solution of
\begin{equation}
\label{eq:spatial_viscosity}
-\left(\isoprofile'' \isoprofile^2 + (\isoprofile')^2 \isoprofile + \int_{\bdry{\Omega_0}} \gausscurv_M\right) = 0
\end{equation}
where $\Omega_0$ is any isoperimetric region corresponding to $a_0$ ($\abs{\Omega_0} = a_0$ and $\isoprofile(a_0) = \abs{\bdry{\Omega_0}}$) and $\gausscurv_M$ is the gauss curvature of $M$.

In particular, if $\gausscurv_M \geq \gausscurv_0$ is bounded below on $M$, then
\[
-\left(\isoprofile'' \isoprofile^2 + (\isoprofile')^2 \isoprofile + \gausscurv_0 \isoprofile\right) \geq 0
\]
in the viscosity sense.
\end{theorem}

\begin{remark}
The integral term in the first equation is difficult to deal with; even though the Gauss curvature $\gausscurv$ is a given function on the ambient space $M$, we don't have any a-priori knowledge of $\Omega_0$.  Nevertheless, the first form will be the most useful to us when considering the Ricci flow, since the integral term will also appear in the time variation of isoperimetric regions under the Ricci flow allowing us to connect the spatial variational formulae with the time variational formulae.
\end{remark}

\begin{proof}
The isoperimetric profile is continuous by proposition \ref{prop:support_iso_nobdry} and the remark following it.

Let $\phi$ be a smooth function defined on a neighbourhood of $a_0 \in (0, \abs{M})$ such that $\phi\leq\isoprofile$ and $\phi(a_0) = \isoprofile(a_0)$.  Let $\Omega_0$ be an isoperimetric region corresponding to $a_0$.  Choose a unit speed normal variation of $\bdry{\Omega_0}$ and define
\[
f(\epsilon) = \abs{\bdry{\Omega_{\epsilon}}} - \phi(\abs{\Omega_{\epsilon}}).
\]
Then we have
\[
f(\epsilon) \geq \isoprofile(\abs{\Omega_{\epsilon}}) - \phi(\abs{\Omega_{\epsilon}}) \geq 0
\]
and
\[
f(0) = \abs{\bdry{\Omega_0}} - \phi(\abs{\Omega_0}) = \isoprofile(\abs{\Omega_0}) - \phi(\abs{\Omega_0}) = 0.
\]
Thus $0$ is a minima of $f$ so that $\inpd[f]{\epsilon}(0) = 0$ and $\inpdd[f]{\epsilon}(0) \geq 0$.  Now we use the first variation formula to compute
\[
\pd[f]{\epsilon} = \int_{\bdry{\Omega_{\epsilon}}} \curvecurv - \phi' \abs{\bdry{\Omega_{\epsilon}}}
\]
which at $\epsilon = 0$ gives
\[
0 = \int_{\bdry{\Omega_0}} \curvecurv - \phi'(a_0) \abs{\bdry{\Omega_0}} = (\curvecurv - \phi'(a_0)) \abs{\bdry{\Omega_0}}
\]
since $\curvecurv$ is constant along $\bdry{\Omega_0}$.  Thus $\curvecurv = \phi'(a_0)$ along $\bdry{\Omega_0}$.

The second variation gives
\[
\begin{split}
0 \leq \pdd[f]{\epsilon} &= \pdd{\epsilon} \abs{\bdry{\Omega_{\epsilon}}} - \phi'' (\pd{\epsilon} \abs{\Omega_{\epsilon}})^2 - \phi' \pdd{\epsilon} \abs{\Omega_{\epsilon}} \\
&= - \int_{\bdry{\Omega_{\epsilon}}} \gausscurv - \phi'' (\abs{\bdry{\Omega_{\epsilon}}})^2 - \phi' \int_{\bdry{\Omega_{\epsilon}}} \curvecurv \\
&= - \int_{\bdry{\Omega_0}} \gausscurv - \phi''(a_0) \phi^2(a_0) - (\phi')^2(a_0) \phi(a_0).
\end{split}
\]
recalling that $\phi(a_0) = \abs{\Omega_0}$ and using that $\curvecurv = \phi'(a_0)$ along $\bdry{\Omega_0}$ as just obtained from the first variation.
\end{proof}
\section{A comparison theorem}
\label{sec-3}
\label{9b6175ce-34f4-40c4-9b36-e7588e0a53d7}

\label{sec:comparison}
\subsection{Comparison equation under the Ricci flow}
\label{sec-3-1}

Let us now couple the spatial viscosity equation with the Ricci flow.  For this we need to know the time-variation of isoperimetric regions under the Ricci flow.  It is quite remarkable that this is possible at all and heavily relies on the fact that $M$ is $2$ dimensional.  It would be interesting to see if similar results hold in higher dimensions, though this seems unlikely unless some topological and/or curvature restrictions are imposed.

We first need the parabolic version of viscosity equations.

\begin{defn}
A lower semi-continuous function $f: (a,b) \times [0, T) \to \RR$ is a viscosity super-solution of the 2nd order parabolic equation
\[
\pd[f]{t} + A(x, t, f, f', f'') = 0
\]
if for every $(x_0, t_0) \in (a,b) \times [0, T)$ and every $C^2$ function $\phi$ such that $\phi(x_0, t_0) = f(x_0, t_0)$ and $\phi(x, t) \leq f(x, t)$ for $x$ in a neighbourhood of $x_0$ and $t\leq t_0$ near $t_0$, we have $\pd[\phi]{t} (x_0, t_0) + A(x_0, t_0, \phi, \phi' , \phi'') \geq 0$.  An upper semi-continuous function is a viscosity sub-solution if the same statements hold with the inequalities reversed.
\end{defn}

\begin{theorem}
\label{thm:ricci_iso_viscosity}
Let $M$ be a closed surface of genus $\genus$, $\metric(t)$ a solution of the normalised Ricci flow on $M$ and $\unicover{\metric}(t) = \uniproj^{\ast}\metric (t)$ the corresponding solution on the universal cover $\uniproj: \unicover{M} \to M$.  For any $a_0$, let $\chi_0$ be the Euler characteristic of $\Omega_0$ an isoperimetric region corresponding to $a_0$.  Then the isoperimetric profile, $\isoprofile_{\unicover{\metric}(t)}$ satisfies
\begin{equation}
\label{eq:temporal_viscosity}
\pd{t}\isoprofile - \left[\isoprofile'' \isoprofile^2 + (\isoprofile')^2 \isoprofile + (4\pi\chi_0 - 2(1-\genus)a)\isoprofile' + (1-\genus) \isoprofile\right] \geq 0
\end{equation}
in the viscosity sense.
\end{theorem}

\begin{proof}
For convenience sake, let us write $\abs{\cdot}_t = \abs{\cdot}_{\unicover{\metric}(t)}$ and $\isoprofile_t = \isoprofile_{\unicover{\metric}(t)}$.  Let $\phi$ be a $C^2$ function such that $\phi(a_0, t_0) = \isoprofile_{t_0}(a_0)$ and $\phi \leq \isoprofile$ for $a$ near $a_0$ and $t\leq t_0$ near $t_0$.  We need to show that $\phi$ satisfies the differential inequality \eqref{eq:temporal_viscosity}.

We compute the time variation of isoperimetric regions.  Given $a_0$, let $\Omega_0\subset \unicover{M}$ be an isoperimetric region in $\unicover{M}$ with respect to the metric $\unicover{\metric}(t_0)$.  That is $\abs{\Omega_0}_{t_0} = a_0$ and $\abs{\bdry{\Omega_0}}_{t_0} = \phi(a_0, t_0)$.  Since $\isoprofile_t(a) \geq \phi(a,t)$ for $t\leq t_0$ and $a$ near $a_0$, we have
\[
\abs{\bdry{\Omega_0}}_t \geq \phi\left(\abs{\Omega_0}_t, t\right)
\]
for $t\leq t_0$, and equality holds when $t=t_0$.  Since both sides of this equation are differentiable in $t$, it follows that under the normalised Ricci flow,
\begin{equation}\label{eq:timeineq1}
\frac{\partial}{\partial t}\Big|_{t=t_0} \abs{\bdry{\Omega_0}}_t \leq \pd[\phi]{t} (a_0,t_0) + \phi'(a_0,t_0) \pd{t}\Big|_{t=t_0} \abs{\Omega_0}_{t}.
\end{equation}
The time derivative on the left can be computed as follows: Parametrise $\bdry{\Omega_0}$ by $\gamma: u \in \sphere^1 \mapsto M$ and write $\gamma_u = \gamma_{\ast} \pd{u}$.  Then recalling that the metric evolves by the normalised Ricci flow, $\pd{t}\unicover{\metric} = -2(\gausscurv - (1-\genus))\unicover{\metric}$, we obtain
\begin{align*}
\pd{t}\Big|_{t=t_0} \abs{\bdry{\Omega_0}} &= \pd{t} \int_{\bdry{\Omega_0}} \sqrt{\unicover{\metric}_t(\gamma_u,\gamma_u)}\,du = -\int_{\bdry{\Omega_0}}(\gausscurv_M - (1 - \genus))ds \\
& = -\int_{\bdry{\Omega_0}} \gausscurv_M\,ds + (1-\genus)\phi (a_0, t_0),
\end{align*}
where $ds$ is the arc-length element along $\bdry{\Omega_0}$.

For the right hand side, by differentiating the determinant and using the normalised Ricci flow equation again, we have $\pd{t} \measure_{\unicover{\metric}} = -2(\gausscurv_M - (1-\genus))\measure_{\unicover{\metric}}$ where $\measure_{\unicover{\metric}}$ is the measure on $\unicover{M}$ induced by the metric $\unicover{\metric}$. Thus,
\[
\pd{t} \Big|_{t=t_0} \abs{\Omega_0}_t = -2\int_{\Omega_0}(\gausscurv_M - (1-\genus) d\measure_{\unicover{\metric}(t_0)}.
\]
Writing $\chi_0 = \chi(\Omega_0)$ the Euler characteristic of $\Omega_0$ and applying the Gauss-Bonnet theorem yields
\[
\pd{t}\Big|_{t=t_0} \abs{\Omega_0}_t = 2(1-\genus)\abs{\Omega_0} - 2\left(2\pi\chi_0 - \int_{\bdry{\Omega_0}} \curvecurv\,ds\right) = 2(1-\genus)a_0 -4\pi\chi_0 +2 \int_{\bdry{\Omega_0}}\curvecurv\,ds,
\]
were $\curvecurv$ is the geodesic curvature of the curve $\bdry{\Omega_0}$.  Thus the inequality \eqref{eq:timeineq1} becomes
\begin{equation}\label{eq:timeineq}
-\int_{\bdry{\Omega_0}}\gausscurv_M\,ds + (1-\genus)\phi \leq \pd{t}\phi + \phi'\left(2(1-\genus)a_0 -4\pi\chi_0 + 2 \int_{\bdry{\Omega_0}}\curvecurv\,ds\right).
\end{equation}

Now recall that theorem \ref{thm:viscosity_nobdry} states that for each time $t$, the isoperimetric profile $\isoprofile_t$ satisfies
\[
-\left(\isoprofile'' \isoprofile^2 + (\isoprofile')^2 \isoprofile + \int_{\bdry{\Omega_0}} \gausscurv_M\right) \geq 0
\]
in the viscosity sense.  Since at $a_0$, $\phi(-, t_0)$ is a supporting function for $\isoprofile_{t_0}(-)$ we also have
\begin{equation}
\label{eq:spaceineq}
\phi'' \phi^2 + (\phi')^2 \phi \leq -\int_{\bdry{\Omega_0}} \gausscurv_M.
\end{equation}
Also, the vanishing of the first spatial variation gives $\curvecurv = \phi'(a_0)$ is constant along $\bdry{\Omega_0}$ and so
\begin{equation}
\label{eq:curvfirstvar}
\int_{\bdry{\Omega_0}} \curvecurv \, ds = \phi(a_0) \phi' (a_0).
\end{equation}

Putting together the inequalities \eqref{eq:timeineq} and \eqref{eq:spaceineq} and using \eqref{eq:curvfirstvar} we obtain
\begin{equation}
\begin{split}
\pd[\phi]{t} &\geq -\int_{\bdry{\Omega_0}}\gausscurv_M\,ds + (1-\genus)\phi - \phi'\left(2(1-\genus)a_0 - 4\pi\chi_0 + 2\phi\phi'\right) \\
&\geq \phi'' \phi^2 + (\phi')^2 \phi + (1-\genus)\phi + \phi'\left(4\pi\chi_0 - 2(1-\genus)a_0\right) - 2\phi(\phi')^2 \\
&= \phi'' \phi^2 - (\phi')^2 \phi + (1-\genus)\phi + \phi'\left(4\pi\chi_0 - 2(1-\genus)a_0\right)
\end{split}
\end{equation}
which is the required inequality.
\end{proof}

\begin{remark}
The viscosity equation includes the $\chi_0$ term which, without any topological knowledge of isoperimetric regions is essentially unknown and could a priori, take on any possible value.  By Corollary \ref{cor:isoregions_positive_curvature}, in the particular case that $\gausscurv_M > 0$, we may conclude that $\chi_0 = 1$ for all $a_0$.  In general however, we need not expect any particular bound on Euler characteristic from a curvature bound alone.
\end{remark}

Even though the topological uncertainty is a real problem, for our purposes we may avoid it entirely by appealing to the underlying concavity of the isoperimetric profile.  This is exploited in the next theorem, the comparison theorem, which is the central result of this paper.

\begin{theorem}
\label{thm:comparison}
Let $(M, \metric(t))$, $(\unicover{M}, \unicover{\metric}(t))$ be as in the previous theorem.  Let $\phi: (0, \abs{\unicover{M}}) \times [0, T) \to \RR$ be a smooth, strictly positive, strictly concave function satisfying
\begin{equation}
\label{eq:ricci_iso_diff_inequal}
\pd[\phi]{t} \leq \phi'' \phi^2 - (\phi')^2 \phi + \phi'\left(4\pi - 2(1-\genus)a_0\right) + (1-\genus)\phi.
\end{equation}
along with the asymptotic behaviour
\[
\limsup_{a\to 0} \frac{\phi(a, t)}{\sqrt{4\pi a}} \leq  1
\]
and
\[
\limsup_{a\to \infty} \left(\isoprofile(a, t) - \phi(a,t)\right) \geq 0
\]

Then if the initial inequality, $\phi(a, 0) < \isoprofile_{\unicover{\metric}(0)} (a)$ for all $a \in (0, \abs{\unicover{M}})$ holds, the inequality $\phi(a, t) \leq \isoprofile_{\unicover{\metric}(t)} (a)$ holds for all $a, t$ with strict inequality if the inequality in \eqref{eq:ricci_iso_diff_inequal} is strict.
\end{theorem}

\begin{remark}
The large scale asymptotic requirements are rather imprecise because we don't have a priori control over the constant $C$ in Theorem \ref{thm:asymptotics_large}. However, this will not prove problematic for us by Proposition \ref{prop:comparison} below.
\end{remark}

\begin{proof}
First suppose that we have strict inequality in the differential inequality and in the asymptotic inequalities. We argue by contradiction.  The conditions $\phi(a, 0) < \isoprofile_{\unicover{\metric}(0)} (a)$ and $\phi(a, t) < \isoprofile_{\unicover{\metric}(t)} (a)$ for $a$ sufficiently close to $\{0, \abs{\unicover{M}}\}$ imply that if the theorem is false, there is a first time $t_0>0$ and an $a_0 \in (0, \abs{\unicover{M}})$ such that $\phi(a_0, t_0) = \isoprofile_{t_0} (a_0)$.  Thus $\phi(a, t) \leq \isoprofile_t(a)$ for $t\leq t_0$ with equality at $(a_0, t_0)$.  Since $\phi$ is strictly concave, the hypotheses of Lemma \ref{thm:curvetopology} are satisfied, so $\Omega_0$ is simply connected and $\chi_0 = 1$.

But now observe that $\phi$ is a lower supporting function for $\isoprofile_t$ at $a_0$ and by theorem \ref{thm:ricci_iso_viscosity},
\[
\pd[\phi]{t} \geq \phi'' \phi^2 - (\phi')^2 \phi + \phi'\left(4\pi - 2(1-\genus)a_0\right) + (1-\genus)\phi
\]
a contradiction, hence the theorem is true when the inequalities are strict.

If any of the inequalities are not strict, define
\[
\phi_{\epsilon} = (1-\epsilon) \phi
\]
for any $\epsilon$ with $0<\epsilon<1$. Then we have $\phi_{\epsilon} < \phi$ giving strict inequality for the asymptotics. We also have
\[
\begin{split}
&\pd[\phi_{\epsilon}]{t} - (\phi_{\epsilon}'' \phi_{\epsilon}^2 - (\phi_{\epsilon}')^2 \phi_{\epsilon}) - \phi_{\epsilon}'\left(4\pi - 2(1-\genus)a_0\right) - (1-\genus)\phi_{\epsilon} \\
&= (1-\epsilon)\left(\pd[\phi]{t} - (1-\epsilon)^2(\phi'' \phi^2 - (\phi')^2 \phi) - \phi'\left(4\pi - 2(1-\genus)a_0\right) - (1-\genus)\phi\right) \\
&\leq \epsilon(1-\epsilon)(2-\epsilon)(\phi^2\phi'' - (\phi')^2\phi) \\
&< 0
\end{split}
\]
since $\phi''< 0$.

Thus $\phi_{\epsilon} (a,t) < \isoprofile (a,t)$ by the result for strict inequalities and the result follows by letting $\epsilon \to 0$.
\end{proof}

\begin{remark}
It's not entirely clear whether strict concavity may be relaxed to merely concavity. A strictly concave approximation to $\phi$ may increase $\phi$ violating the inequality $\phi \leq \isoprofile$.
\end{remark}

Using the theorem, and the asymptotics of $\isoprofile$ from Theorem \ref{thm:bdrybehaviour},
\[
\isoprofile(a) = \sqrt{4 \pi a}(1 - \frac{\sup_M\gausscurv}{8\pi} a + O(a^2))
\]
as $a\to 0$, we may now obtain a curvature bound for $\unicover{\metric}(t)$ and hence for $\metric(t)$.

\begin{cor}
\label{cor:ricci_comp_curv_bnd}
With the notation of the previous theorem, $\phi$ satisfying the hypothesis of the theorem and such that
\[
\phi (a, t) = \sqrt{4\pi a}(1 - \frac{\gausscurv_0(t)}{8\pi} a + O(a^2)),
\]
we have
\[
\sup_M \gausscurv_M(t) \leq \gausscurv_0(t).
\]
\qed
\end{cor}

The isoperimetric constant of a non-compact surface is defined to be
\[
\isoconst = \inf\left\{ \frac{\abs{\bdry{\Omega}}^2}{\abs{\Omega}} : \Omega \quad \text{admissible} \right\} = \inf \left\{\frac{\isoprofile(a)^2}{a} : 0 < a < \infty \right\}.
\]
For a compact surface, the (modified) isoperimetric constant is defined by
\[
\isoconst = \inf\left\{ \frac{\abs{\bdry{\Omega}}^2}{\min(\abs{\Omega}, \abs{M \setdiff \Omega})} : \Omega \quad \text{admissible} \right\} = \inf \left\{ \frac{\isoprofile(a)^2}{a} : 0 < a < \frac{\abs{M}}{2} \right\}.
\]

\begin{cor}
\label{cor:isoconst_bdd}
With the notation of the previous theorem and, $\phi$ satisfying the hypothesis of the theorem we have
\[
\isoconst_{\unicover{M}}(t) \geq \inf\left\{\frac{\phi(a, t)^2}{a} : 0 < a < \frac{\abs{\tilde{M}}}{2}\right\}.
\]
\end{cor}

\begin{remark}
\label{rem:isoconst_bdd}
Note that area (2 dimensional volume) on $\unicover{M}$ equipped with the pull-back metric $\pullback{\uniproj} \metric$ grows like the growth of the fundamental group, but boundary length can't be controlled so easily. For instance, the torus with arbitrarily small ratio of principal radii may be equipped with the flat metric, giving control of the isoperimetric constant on $\unicover{M}$, but with arbitrarily small isoperimetric constant on $M$. Note however, that by normalising the area of $M$ to $4\pi$, we avoid this issue, and I conjecture that the isoperimetric constant of $M$ may be bounded below by that of $\unicover{M}$ for any $\genus \geq 1$. For the matter at hand, when $\genus>0$ (so that $\unicover{M}$ is not compact), we can't immediately transfer control of the isoperimetric constant on $\unicover{M}$ to control of the isoperimetric constant on $M$.
\end{remark}

Let us finish this section by recording a useful result for surfaces of genus $\genus \geq 1$ that shows the large scale asymptotics of $\phi$ are superfluous.

\begin{prop}
\label{prop:comparison}
Let $M$ be a closed surface of genus $\geq 1$ (so that $\unicover{M}$ is not compact).  Let $\phi$ be a strictly positive, strictly concave function satisfying the differential inequality \eqref{eq:ricci_iso_diff_inequal} and the small scale asymptotics from the comparison theorem \ref{thm:comparison}. Then if $\phi(a, 0) < \isoprofile_{\unicover{\metric}(0)}(a)$ for all $a\in (0,\infty)$, then $\phi(a, t) \leq \isoprofile_{\unicover{\metric}(t)} (a)$ for all $a,t$.
\end{prop}

\begin{proof}
The only thing missing from Theorem \ref{thm:comparison} is the large scale asymptotics. It is convenient to work with the function $v=\phi^2$.  This satisfies
\begin{equation}
\label{eq:ricci_isosq_diff_inequal}
\pd[v]{t} \leq v^2 \laplace \ln v + (4\pi - 2(1-\genus)a) v' + 2(1-\genus) v.
\end{equation}
For any $C>0$, define 
\[
u_C(a, t) = C e^{2(1-\genus)t}.
\]
Then $u_C$ satisfies equality in equation \eqref{eq:ricci_isosq_diff_inequal}.  Since $u_C$ is constant for each fixed $t$ and $\isoprofile$ grows at least linearly as $a\to\infty$ by Theorem \ref{thm:asymptotics_large}, we have also have $u_C(a, t) < \isoprofile_{\unicover{\metric}(t)}(a)$ for all $a$ large enough.  Now take the harmonic mean,
\[
\harmean (a, t) = \left(\frac{1}{v(a,t)} + \frac{1}{u_C(a,t)}\right)^{-1}.
\]
This has the property that for any $(a,t)$ we have
\[
v(a, t) = \lim_{C\to\infty} \frac{v(a, t) u_C(a, t)}{v(a, t) + u_C(a, t)} = \lim_{C\to\infty} \harmean (a, t).
\]
Therefore to prove the result, we need to show $\harmean$ satisfies the hypotheses of theorem \ref{thm:comparison} since this will give the inequality for $\harmean$ for every $C>0$ and so too for $v$ being the limit $C\to\infty$ of $\harmean$.

First, since $0 < \harmean \leq v, u_C$, the initial inequality $\harmean < \isoprofile_0$ is satisfied along with the necessary small and large scale asymptotics. 

For strict concavity of $\harmean$, we use 
\[
\harmean = \frac{u_C v}{u_C + v}
\]
and $(u_C)' = 0$ to compute
\begin{align*}
\harmean' &= \frac{u_Cv'}{u_C + v} - \frac{u_C v v'}{(u_C + v)^2} \\
&= \frac{u_C^2 v'}{(u_C + v)^2}
\end{align*}
and so
\[
\harmean'' = \frac{u_C^2 v''}{(u_C + v)^2} - \frac{2u_C^2(v')^2}{(u_C + v)^3} < 0
\]
by strict concavity of $v$ and positivity of $v, u_C$.  Thus $H$ is strictly concave.  

Now let us consider the differential inequality.  Define
\[
L_{\pm} = \left(4\pi - 2(1-\genus)a\right)\pd{a} \pm 2(1-\genus).
\]
The differential inequality \eqref{eq:ricci_isosq_diff_inequal} then reads
\[
\left(\pd{t} - L_-\right) v \leq v^2 \laplace \ln v.
\]
For any function $f$ we have
\begin{equation}
\label{eq:linopf}
\left(\pd{t} - L_{\pm}\right) \frac{1}{f} = -\frac{1}{f^2} \left(\pd{t} - L_{\mp}\right) f.
\end{equation}
Applying equation \eqref{eq:linopf} to $\harmean = 1/f$ with $f = v^{-1} + u_C^{-1}$ gives
\begin{align*}
\left(\pd{t} - L_-\right) \harmean &= -\harmean^2 \left((\pd{t} - L_+) \frac{1}{v} + (\pd{t} - L_+) \frac{1}{u_C}\right) \\
&= -\harmean^2 \left((\pd{t} - L_+) \frac{1}{v} - 2(1-\genus)\frac{1}{u_C}\right)
\end{align*}
since $L_{\pm} (u_C) = 0$ and $u_C$ satisfies equality in \eqref{eq:ricci_isosq_diff_inequal}.

Next applying equation \eqref{eq:linopf} to $v^{-1}$ we get
\begin{align*}
\left(\pd{t} - L_-\right) \harmean &= \frac{\harmean^2}{v^2} (\pd{t} - L_-) v + 2\harmean^2 (1-\genus)\frac{1}{u_C} \\
&\leq \frac{\harmean^2}{v^2} v^2 \laplace \ln v \\
&= \harmean^2 \laplace \ln v.
\end{align*}
since $(1-\genus) \leq 0$. Here the inequality is strict if $v$ (or equivalently $u$) satisfies strict inequality in the differential inequality.  We want to show the right hand side is less than or equal to $\harmean^2 \laplace \ln \harmean$.  We compute
\begin{align*}
\harmean^2 \laplace \ln \harmean &= \harmean \harmean'' - (\harmean')^2 \\
&= \frac{v u_C}{v+u_C} \left[\left(\frac{u_C}{v+u_C}\right)^2 v'' - \frac{2u_C^2}{(v+u_C)^3} (v')^2\right] - \left(\frac{u_C}{v+u_C}\right)^4 (v')^2 \\
&= \left(\frac{v u_C}{v + u_C}\right)^2 \left[\left(\frac{u_C}{v+u_C}\right) \frac{v''}{v} - \frac{2u_C v}{(v+u_C)^2} \frac{(v')^2}{v^2} - \frac{u_C^2}{(v+u_C)^2} \frac{(v')^2}{v^2}\right] \\
&= \harmean^2 \left[\left(\frac{u_C}{v+u_C}\right) \frac{v''}{v} - \left(\frac{(v+u_C)^2 - v^2}{(v+u_C)^2}\right) \frac{(v')^2}{v^2}\right] \\
&\geq \harmean^2 \left[ \frac{v''}{v} - \frac{(v')^2}{v^2}\right] = \harmean^2 \laplace \ln v
\end{align*}
where the inequality follows from the concavity of $v$ and the positivity of $v$ and $u_C$.
\end{proof}
\subsection{A connection with logarithmic porous media}
\label{sec-3-2}

For positive functions $\phi$, the differential inequality
\[
\pd[\phi]{t} < \phi^2\phi'' - \phi(\phi')^2 + (4\pi - 2(1-\genus)a)\phi' + (1-\genus)\phi
\]
is equivalent to the logarithmic porous media inequality
\[
\pd[u]{t} > \laplace \ln u.
\]
To see this, observe that
\[
\phi^3 \laplace \ln \phi = \phi^2\phi'' - \phi(\phi')^2.
\]
Letting $u = \phi^{-2}$ we have $\laplace \ln u = -2 \laplace \ln \phi$ and so
\[
\begin{split}
\pd[u]{t} &= \frac{-2}{\phi^3} \pd[\phi]{t} > \frac{-2}{\phi^3} \left[\phi^3 \laplace \ln \phi + (4\pi - 2(1-\genus)a)\phi' + (1-\genus)\phi\right] \\
&= \laplace \ln u + (4\pi - 2(1-\genus)a)u'  -2(1-\genus)u.
\end{split}
\]
A change of the independent variables $(a,t)$ can now be made to get rid of the lower order terms.  This point of view may prove useful since the logarithmic porous media equation has been extensively studied, but we do not use it here.
\section{Model solutions}
\label{sec-4}
\label{396a40e0-d4d8-422d-a69b-4e25da6de9c3}

\label{sec:models}

This section is devoted to exhibiting suitable comparison functions $\phi$ and hence curvature and isoperimetric bounds for metrics evolving by the normalised Ricci flow via Corollaries \ref{cor:ricci_comp_curv_bnd} \ref{cor:isoconst_bdd}. We will need to treat the cases $\genus=0, \genus=1, \genus>1$ separately.  The next and final section briefly outlines how such bounds lead, via standard arguments, to the convergence results described in Theorem \ref{thm:convergence}.
\subsection{genus 0}
\label{sec-4-1}

In \cite{MR2729306}, we showed that the isoperimetric profile of the Rosenau solution provided a suitable comparison solution.  Let us briefly recall the result.  The Rosenau solution is an explicit axially symmetric solution of the normalized Ricci flow on the two-sphere.  The metric is given by $\bar g(t) = u(x,t)(dx^2+dy^2)$, where $(x,y)\in\RR\times[0,4\pi]$, and
\[
u(x,t) = \frac{\sinh(e^{-2t})}{2e^{-2t}\left(\cosh(x)+\cosh(e^{-2t})\right)}.
\]
This extends to a smooth metric on the two-sphere at each time with area $4\pi$, and which evolves according to the normalized Ricci flow equation \eqref{eq:NRF}.  A direct computation gives the isoperimetric profile,
\begin{equation}
\label{eq:Rosenauprofile}
\varphi(a,t) = \sqrt{4\pi}\sqrt{\frac{\sinh(a e^{-2t})\sinh((1-a)e^{-2t})}{\sinh(e^{-2t})e^{-2t}}}.
\end{equation}

By translating $t \mapsto t - t_0$ with $t_0$ chosen so that initial inequality of the isoperimetric profile holds, the comparison theorem leads to the following bounds for solutions of the normalised Ricci flow on the $2$-sphere:

\begin{theorem}
\label{thm:sphere_bounds}
Let $\metric$ be a solution of the normalised Ricci flow on $\sphere^2$.  Then there exists constants $A,C>0$ depending only on the metric at the initial time such that 
\[
\sup_{\sphere^2} \gausscurv (t) \leq C e^{-At}.
\]
There also exists a constant $\isoconst_0 > 0$, depending only on the initial metric $\metric_0$, such that
\[
\isoconst (t) > \isoconst_0
\]
where $\isoconst(t)$ is the isoperimetric constant of $(\sphere^2, \metric(t))$.
\end{theorem}
\subsection{genus 1}
\label{sec-4-2}

Next, let us describe a comparison solution for the universal cover of surfaces of genus $\genus = 1$, i.e. for $\RR^2$. 

Recall, we need to find a function satisfying the differential inequality
\[
\phi_t\geq \phi^2\phi''-\phi(\phi')^2+4\pi \phi'.
\]
We look for solutions with equality.  First, to simply matters, let $v=\phi^2$ which satisfies the  equation
\[
v_t = vv''-(v')^2+4\pi v' = v^2\left(\frac{v'}{v} - \frac{4\pi}{v}\right)'.
\]
Taking the Ansatz $v(a,t) = tV(a/t)$, we obtain an integrable equation, which adding in the limiting behaviour $V(0) = 0$, has the family of solutions
\[
V_C(z) = \frac{1}{C}\left(4\pi - \frac{1}{C}\right) \left(1 - e^{-Cz}\right) + \frac{z}{C}.
\]
That is, we have
\begin{equation}
\label{eq:ricci_plane_comparison}
v_C(a,t) = \frac{a}{C} + \frac{t}{C}\left(4\pi - \frac{1}{C}\right) \left(1 - e^{-\tfrac{Ca}{t}}\right).
\end{equation}

We can now use $v_C$ as a comparison for $\genus = 1$ surfaces, as in the following:

\begin{theorem}
\label{thm:plane_comparison}
Let $\metric(t)$ be any solution of the normalised Ricci flow on $M$ a closed, genus $1$ surface and let $\unicover{\metric} = \uniproj^{\ast}\metric$ be the pull back metric to the universal cover $\unicover{M} = \RR^2$.  Then there exists a $C>0$ such that the function $\phi = \sqrt{v_c}$ where $v_c$ is defined by \eqref{eq:ricci_plane_comparison} satisfies $\phi(a, t) < \isoprofile_{\unicover{\metric}} (a, t)$ for all $a \in (0,\infty)$ and $t \in [0, T)$. Therefore the Gauss curvature $\gausscurv$ of $M$ satisfies the bound
\[
\sup_M \gausscurv \leq \frac{A}{t}
\]
for a constant $A>0$ depending only on the initial metric $\metric_0$.
\end{theorem}

\begin{proof}
We know that $v_C$ satisfies the differential inequality and it's easy to see that $v_C$ is strictly concave, so we need to show that $v_C$ meets the other requirements for the comparison theorem in the form of Proposition \ref{prop:comparison}.  At $t=0$, we have $v_C(a,0) = \tfrac{a}{C}$ so by choosing $C$ large enough, we have the initial comparison since $\isoprofile \simeq \sqrt{C_1 a + C_2 a^2}$ as $a\to \infty$. 

On the small scale we have 
\[
\phi(a, t) = \sqrt{4\pi a} \left(1 - (\frac{C}{4} - \frac{1}{16\pi})\frac{1}{t} a + O(a^2)\right)
\]
as required for the small scale asymptotics and also providing the stated curvature bound with $A = 2\pi C - 1/2$ (which is positive for $C > 1/4\pi$) by Corollary \ref{cor:ricci_comp_curv_bnd}.
\end{proof}
\subsection{genus $>1$}
\label{sec-4-3}

In this section, we construct the model comparison solution for the final case, $\genus > 1$. When $\sup_{M_0} \gausscurv > 0$, the construction is a little involved.
\subsubsection{$K < 0$ case}
\label{sec-4-3-1}

First let us consider the case where \(\sup_{M_0} \gausscurv \leq 0\), since it admits a simple, appealing comparison solution.

For any $A,C>0$, let
\begin{equation}
\label{eq:ricci_hyperbolic_comparison}
v(a, t) = 4\pi a + B(t) a^2
\end{equation}
with
\[
B(t) = (\genus - 1) - \frac{C}{1 + A e^{(\genus - 1)t}}.
\]
Direct computation shows that $v_C$ is a solution of the differential equation
\[
v_t = vv''-(v')^2 + (4\pi - (1-\genus)a) v' + 2(1-\genus)v,
\]
which is the required equation for $v = \phi^2$ as in the genus $1$ case above.

\begin{theorem}
\label{thm:ricci_hyperbolic_curvature_bound_negative}
Let $\metric(t)$ be any solution of the normalised Ricci flow on $M$ a closed, genus $>1$ surface with $\sup_{M_0} \gausscurv \leq 0$ and let $\unicover{\metric} = \uniproj^{\ast}\metric$ the pull back to $\HH^2$ with $\uniproj: \HH^2 \to M$ the universal cover.  Then for $\phi = \sqrt{v}$ where $v$ is defined by \eqref{eq:ricci_hyperbolic_comparison}, there exists $A,C>0$ such that $\phi(a, t) < \isoprofile_{\unicover{\metric}} (a, t)$ for all $a \in (0,\infty)$.  Therefore, the Gauss curvature $\gausscurv_M$ satisfies the bound
\[
\sup_{M_t} \gausscurv \leq C_1 e^{-C_2t}
\]
for positive constants $C_1,C_2$.
\end{theorem}

\begin{proof}
Since the comparison function is a quadratic with zero constant term and linear coefficient equal to $4\pi$, the small scale asymptotics are satisfied providing that $B(t) \leq - \text{const} \sup_M\gausscurv$, by the asymptotics of the isoperimetric profile given in theorem \ref{thm:bdrybehaviour}. Since we require $B(t)\geq 0$, this can only be achieved in the case  $\sup_{M_t} \gausscurv \leq 0$ which is true by the maximum principle under the assumption $\sup_{M_0} \gausscurv \leq 0$. In this case, we choose $A,C$ large enough so that the initial comparison holds. Concavity is easily checked. Proposition \ref{prop:comparison} completes the proof that $\phi(a, t) < \isoprofile_{\unicover{\metric}} (a, t)$ for all $a \in (0,\infty)$. The curvature bound now follows directly from Corollary \ref{cor:ricci_comp_curv_bnd}.
\end{proof}
\subsubsection{General case}
\label{sec-4-3-2}
\paragraph{Stationary solution}
\label{sec-4-3-2-1}

Recall we have the equation,
\[
\pd{t} v - \left\{vv'' - (v')^2 + [4\pi - 2(1-\genus)a]v' + 2(1-\genus)v\right\} \leq 0.
\]
We can write this as
\begin{equation}
\label{eq:hyperbolic_squared_diff_ineq}
\pd{t} v \leq v^2 \left(\frac{v'}{v} - \frac{4\pi - 2(1-\genus)a}{v}\right)'.
\end{equation}
Stationary solutions (with equality) to this equation that satisfy the conditions $v(0)=0$ and $\limsup_{x\to\infty} \tfrac{v(x)}{x^2} < \infty$ are given by
\begin{equation}
\label{eq:hyperbolic_stationary}
\begin{split}
v_C(x) &= \frac{1}{C}[4\pi + \frac{2(1 - \genus)}{C}][1 - e^{-Cx}] - \frac{2(1-\genus)}{C^2} (Cx) \\
&= 4\pi x + \frac{1}{C}[4\pi + \frac{2(1 - \genus)}{C}][1 - Cx - e^{-Cx}] \\
&= 4\pi x - \frac{1}{C}[4\pi + \frac{2(1 - \genus)}{C}]\frac{(Cx)^2}{2} \\
& \quad + \frac{1}{C}[4\pi + \frac{2(1 - \genus)}{C}][1 - Cx + \frac{(Cx)^2}{2} - e^{-Cx}]
\end{split}
\end{equation}
for any $C\geq 0$. The last line is obtained from the Taylor expansion for $e^{-Cx}$. Each of the three expressions illustrates different properties of $v_C$. For instance, the first line shows that $v_C$ grows at most linearly. The second and third lines give the first and second order Taylor expansions with explicit remainders.

For later use, the first and second derivatives of $v_C$ are
\begin{align}
\label{eq:hyperbolic_stationary_1st}
v_C' &= 4\pi + [4\pi + \frac{2(1 - \genus)}{C}][-1 + e^{-Cx}] \\
\label{eq:hyperbolic_stationary_2nd}
v_C'' &= -C[4\pi + \frac{2(1 - \genus)}{C}]e^{-Cx}.
\end{align}
In particular, provided that 
\[
C \geq C_{\text{crit}} = - \frac{1-\genus}{2\pi}
\]
we have $[4\pi + \frac{2(1 - \genus)}{C}] \geq 0$ and so $V_C$ is concave, strictly so when $C>C_{\text{crit}}$.

Such functions prove useful, but are not quite sufficient for our purposes. The comparison is constructed from the function,
\begin{equation}
\label{eq:f_hyperbolic}
f(x, t) = \sqrt{v_{C} (x) + b x^2}
\end{equation}
with $b\geq 0$.

\begin{lemma}
\label{lem:hyperbolic_stationary_concave}
Let $f$ be defined as in equation \eqref{eq:f_hyperbolic} with $C\geq C_{\text{crit}}$. Then $f$ is concave, if and only if 
\[
b \leq b_{\text{crit}} = \frac{(1-\genus)^2}{\frac{1}{C}[4\pi + \frac{2(1-\genus)}{C^2}]},
\]
with strict concavity corresponding to strict inequality.
\end{lemma}

\begin{proof}
We have
\[
f'' = \frac{1}{2f^3}[v_Cv_C'' - \frac{1}{2}(v_C')^2 + b(2v_c - 2xv_C' + x^2 v_C'')]
\]
so that $f$ is concave if and only if
\[
b \leq \frac{\frac{1}{2}(v_C')^2 - v_Cv_C''}{(2v_c - 2xv_C' + x^2 v_C'')}
\]
since $b$ is non-negative.

First, consider the numerator. Since $v_C\geq 0$ and $v_C''\leq 0$ we have 
\[
\frac{1}{2}(v_C')^2 - v_C v_C'' \geq \frac{1}{2}(v_C')^2.
\]
Again using $v_C''\leq 0$ we have
\[
v_C'(x) \geq \lim_{x\to\infty} v_C' = \frac{-2(1-\genus)}{C}
\]
Also $\lim_{x\to\infty} v_Cv_C'' = 0$ so that in fact,
\[
\inf_{x}\left(\frac{1}{2}[(v_C')^2 - v_Cv_C'']\right) = \frac{2(1-\genus)^2}{C^2}.
\]

For the denominator, we have
\begin{align*}
2v_C - 2xv_C' + x^2v_C &= \frac{2}{C}[4\pi + \frac{2(1 - \genus)}{C}][1 - e^{-Cx}(1 + Cx + (Cx)^2/2)] \\
&\leq \frac{2}{C}[4\pi + \frac{2(1 - \genus)}{C}]
\end{align*}
with equality at $x=0$ so that
\[
\sup_{x} \left(2v_C - 2xv_C' + x^2v_C\right) = \frac{2}{C}[4\pi + \frac{2(1 - \genus)}{C}].
\]
Therefore $f$ is concave if and only if
\[
b \leq \frac{\frac{2(1-\genus)^2}{C^2}}{\frac{2}{C}[4\pi + \frac{2(1 - \genus)}{C}]} = \frac{(1-\genus)^2}{\frac{1}{C}[4\pi + \frac{2(1-\genus)}{C^2}]}.
\]
\end{proof}

\begin{remark}
\label{rem:hyperbolic_stationary_concave}
Observe that the denominator in $b_{\text{crit}}$ is zero for $C=C_{\text{crit}}$, is positive for $C>C_{\text{crit}}$ and, approaches $0$ as $C\to\infty$. Thus for $C_0 \geq C_{\text{crit}}$, $b_{\text{crit}} ([C_{\text{crit}}, C_0])$ is bounded below away from $0$. This will prove useful later.
\end{remark}

Let us also record the small and large scale asymptotics of $f$ in a lemma for later reference.

\begin{lemma}
\label{lem:hyperbolic_asymptotics}
The function $v_C$ satisfies the asymptotic behaviour
\[
v_C (x) = 4\pi x - \frac{1}{C}\left[4\pi + \frac{2(1-\genus)}{C}\right] \frac{(Cx)^2}{2} + \bigo((Cx)^3))
\]
as $Cx\to 0$. Moreover,
\[
\limsup_{x\to\infty} v_C(x) = -\frac{2(1-\genus)}{C^2} (Cx).
\]

Therefore, $f$ satisfies the asymptotic behaviour
\[
f^2(x) = 4\pi x +\left(\frac{2b}{C^2} - \frac{1}{C}\left[4\pi + \frac{2(1-\genus)}{C}\right]\right) \frac{(Cx)^2}{2} + \bigo((Cx)^3))
\]
as $Cx\to 0$. Moreover
\[
\limsup_{x\to \infty} f^2 = b x^2.
\]
\qed
\end{lemma}

\begin{remark}
\label{rem:hyperbolic_asymptotics}
In particular notice that the coefficient of $x^2/2$ (rather than $(Cx)^2/2$) from the small-scale asymptotics of $f$ is
\[
2b - 4\pi C - 2(1-\genus)
\]
and this can be made arbitrarily large negative by choosing $0 \leq b \ll 1-\genus$ and $C \gg C_{\text{crit}}$.
\end{remark}
\paragraph{Construction of the comparison function}
\label{sec-4-3-2-2}

The comparison is built from the function $f$ defined in equation \eqref{eq:f_hyperbolic} by letting $C=C(t), b=b(t)$. If $C(t) \searrow C_{\text{crit}}$ and $b(t) \nearrow -(1-\genus)$ as $t\to \infty$, with $b(t) \leq b_{\text{crit}}(C(t))$ (which choice is possible by remark \ref{rem:hyperbolic_stationary_concave}), then
\begin{equation}
\label{eq:hyperbolic_comparison}
f(x,t) = \sqrt{v_{C(t)} + b(t)x^2}
\end{equation}
is a concave function with
\[
\lim_{t\to\infty} f(x,t) = \sqrt{4\pi x + (1-\genus)x^2}
\]
the isoperimetric profile of the metric of constant curvature $1-\genus$ (which is the curvature of the metric lifted from the constant curvature surface with area $4\pi$ and genus $\genus$ by the Gauss-Bonnet theorem).

By choosing $C(0)>C_{\text{crit}}$ sufficiently large and $0\leq b(0) < \max\{b_{\text{crit}}, 1-\genus\}$, sufficiently small, lemma \ref{lem:hyperbolic_asymptotics} and remark \ref{rem:hyperbolic_asymptotics} imply that for any initial metric, initial inequality is satisfied along with the asymptotic behaviour required by comparison theorem in the form of Proposition \ref{prop:comparison}.

Thus for $f$ to be a suitable comparison function, we need to choose $C(t)$ and $b(t)$ so that the differential inequality is satisfied. As before, it is more convenient to work with $v=f^2 = v_c + bx^2$

\begin{lemma}
\label{lem:hyperbolic_diff_ineq}
Let 
\begin{align*}
b(t) &= \left[\left(\frac{1}{b_0} + \frac{1}{1-\genus}\right) e^{4(1-\genus)t} - \frac{1}{1-\genus}\right]^{-1} \\
C(t) &= (C_0 - C_{\text{crit}})\sqrt{b_0} e^{2(1-\genus)t} \left[\left(\frac{1}{b_0} + \frac{1}{1-\genus}\right) e^{4(1-\genus)t} - \frac{1}{1-\genus}\right]^{-1/2} + C_{\text{crit}}
\end{align*}
Then $v = f^2$ satisfies the differential inequality \eqref{eq:hyperbolic_squared_diff_ineq} with $f$ defined by \eqref{eq:hyperbolic_comparison}.
\end{lemma}

\begin{proof}
First, for the time derivative we have
\begin{equation}
\label{eq:dt_hyperbolic_stationary}
\begin{split}
\pd[v]{t} &= \dd[C]{t} \pd[v_C]{C} + \dd[b]{t} x^2 \\
&= -\dd[C]{t} \frac{1}{C^2}[4\pi + \frac{2(1-\genus)}{C}] \frac{(Cx)^2}{2} \\
&\quad - \dd[C]{t} \left(\frac{1}{C^2}[4\pi + \frac{4(1-\genus)}{C}] \right) \left(1 - Cx + \frac{(Cx)^2}{2} - e^{-Cx}\right) \\
&\quad - \dd[C]{t} \left(\frac{1}{C^2}[4\pi + \frac{2(1-\genus)}{C}]\right)\left(Cx\right)\left(1 - Cx - e^{-Cx}\right) \\
&\quad + \frac{2}{C^2} \dd[b]{t} \frac{(Cx)^2}{2} \\
&< \frac{1}{C^2} \left(2\dd[b]{t} - \dd[C]{t} [4\pi + \frac{2(1-\genus)}{C}]\right) \frac{(Cx)^2}{2} \\
&\quad - \dd[C]{t} \left(\frac{1}{C^2}[4\pi + \frac{2(1-\genus)}{C}] \right) \left(1 - Cx + \frac{(Cx)^2}{2} - e^{-Cx}\right) \\
&\quad - \dd[C]{t} \left(\frac{1}{C^2}[4\pi + \frac{2(1-\genus)}{C}\right)\left(Cx\right)\left(1 - Cx - e^{-Cx}\right) \\
\end{split}
\end{equation}
The inequality occurs in the second line after the inequality by replacing the $4$ with a $2$ using the fact that $C$ is decreasing and $(1-\genus)<0$.

For the spatial part, we first use the fact that for two functions $g,h$ we have
\[
(g+h)^2 \laplace \ln (g+h) = (g+h)(g+h)'' - (g+h)' = g^2\laplace \ln g + h^2 \laplace \ln h + gh'' + hg'' - 2g'h'
\]
Thus with $g=v_C$ and $h=bx^2$ we get
\[
v^2 \laplace \ln v = v_C^2 \laplace \ln v_C - 2b^2x^2 + b (2v_C - 4 x v_C' + x^2 v_C'') 
\]
so that
\[
\begin{split}
v^2 \laplace \ln v + L[v] &= v_C^2 \laplace\ln v_C + L[v_C] + 8\pi b x - 2(b^2 + (1-\genus)b)x^2 \\
&\quad + b (2v_C - 4 x v_C' + x^2 v_C'') \\
&= 8\pi b x - 2(b^2 + (1-\genus)b)x^2 + b (2v_C - 4 x v_C' + x^2 v_C'')
\end{split}
\]
since $v_C$ satisfies $v_C^2 \laplace\ln v_C + L[v_C] = 0$. Expand the last term in parenthesis in a Taylor series using equations \eqref{eq:hyperbolic_stationary}, \eqref{eq:hyperbolic_stationary_1st} and \eqref{eq:hyperbolic_stationary_2nd} to get
\begin{equation}
\label{eq:spatial_hyperbolic_stationary}
\begin{split}
v^2 \laplace \ln v + L[v] &= \left(\frac{4b}{C}[4\pi + \frac{2(1-\genus)}{C}] - \frac{4(b^2 + (1-\genus)b)}{C^2} \right) \frac{(Cx)^2}{2} \\
&\quad + \frac{2b}{C}[4\pi + \frac{2(1-\genus)}{C}] \left(1 - Cx + \frac{(Cx)^2}{2} - e^{-Cx}\right) \\
&\quad + \frac{4b}{C}[4\pi + \frac{2(1-\genus)}{C}] \left(Cx\right) \left(1 - Cx - e^{-Cx}\right) \\
&\quad + \frac{2b}{C}[4\pi + \frac{2(1-\genus)}{C}] \frac{(Cx)^2}{2} \left(1 - e^{-Cx}\right).
\end{split}
\end{equation}

Now we compare the terms from equation \eqref{eq:dt_hyperbolic_stationary} with those of equation \eqref{eq:spatial_hyperbolic_stationary} to obtain the following necessary inequalities:

\begin{itemize}
\item $(Cx)^2/2$:
  \begin{multline*}
  \frac{1}{C^2} \left(2\dd[b]{t} - \dd[C]{t} [4\pi + \frac{2(1-\genus)}{C}]\right) \\
  < \left(\frac{4b}{C}[4\pi + \frac{2(1-\genus)}{C}] - \frac{4(b^2 + (1-\genus)b)}{C^2} \right)
  \end{multline*}
\item $1 - Cx + \frac{(Cx)^2}{2} - e^{-Cx}$:
  \[  
  -\dd[C]{t} \left(\frac{1}{C^2}[4\pi + \frac{2(1-\genus)}{C}] \right) 
  < \frac{2b}{C}[4\pi + \frac{2(1-\genus)}{C}]
  \]
\item $Cx \left(1 - Cx - e^{-Cx}\right)$:
  \[
  -\dd[C]{t} \left(\frac{1}{C^2}[4\pi + \frac{2(1-\genus)}{C}\right)
  < \frac{4b}{C}[4\pi + \frac{2(1-\genus)}{C}]
  \]
\item $\frac{(Cx)^2}{2} \left(1 - e^{-Cx}\right)$:
  \[
  0 < \frac{2b}{C}[4\pi + \frac{2(1-\genus)}{C}] 
  \]
\end{itemize}

All the above inequalities are satisfied if
\begin{align}
\label{eq:b_bernoulli_hyperbolic_stationary}
\dd[b]{t} &< -4\left(b^2 + (1-\genus)b\right) \\
\label{eq:c_hyperbolic_stationary}
\dd{t} \ln C &> - 2b
\end{align}
and $C \geq C_{\text{crit}}$ which ensures that $4\pi + \tfrac{2(1-\genus)}{C} \geq 0$. 

It is now a simple matter, left to the reader, to check that $b(t)$ as given in the statement of the lemma satisfies equality in the Bernoulli equation \eqref{eq:b_bernoulli_hyperbolic_stationary} and that equality in \eqref{eq:c_hyperbolic_stationary} is satisfied by
\[
\tilde{C} = (C_0 - C_{\text{crit}})\sqrt{b_0} e^{2(1-\genus)t} \left[\left(\frac{1}{b_0} + \frac{1}{1-\genus}\right) e^{4(1-\genus)t} - \frac{1}{1-\genus}\right]^{-1/2}.
\]
Since $C(t) = \tilde{C} + C_{\text{crit}} > \tilde{C}$ and $\dd{t} \tilde{C} < 0$, we then have
\[
\dd{t} \ln C = \frac{\dd{t}\tilde{C}}{\tilde{C} + C_{\text{crit}}} > \frac{\dd{t}\tilde{C}}{\tilde{C}} = -2b
\]
competing the proof.
\end{proof}

\begin{remark}
Observe that with $b(t), C(t)$ as given in lemma \ref{lem:hyperbolic_diff_ineq}, $b(t)$ monotonically increases from $b_0$ to $-(1-\genus)$ and $C(t)$ monotonically decreases from $C_0$ to $C_{\text{crit}}$ so that $f = \sqrt{v}$ converges to the constant curvature $1-\genus$ isoperimetric profile.
\end{remark}

Finally, applying corollary \ref{cor:ricci_comp_curv_bnd} we obtain
\begin{theorem}
\label{thm:ricci_hyperbolic_curvature_bound_general}
Let $\metric(t)$ be any solution of the normalised Ricci flow on $M$ a closed, genus $>1$ surface and let $\unicover{\metric} = \uniproj^{\ast}\metric$ the pull back to $\HH^2$ with $\uniproj: \HH^2 \to M$ the universal cover.  Then for $\phi = \sqrt{v_{C(t)} + b(t)a^2}$ where $C(t), b(t)$ are defined as in Lemma \ref{lem:hyperbolic_diff_ineq} there exists $C_0, b_0>0$ such that $\phi(a, t) < \isoprofile_{\unicover{\metric}} (a, t)$ for all $a \in (0,\infty)$.  Therefore, the Gauss curvature $\gausscurv_M$ satisfies the bound
\[
\sup_{M} \gausscurv_{M} (t) \leq \left(\frac{C(t)}{2}\left[4\pi + \frac{2(1-\genus)}{C(t)}\right] - b(t)\right)
\]
which decays exponentially fast to $(1-\genus)$ as $t \to \infty$.
\end{theorem}

\begin{remark}
The exponential decay in the theorem follows from the fact that $b(t) \to -(1-g)$ exponentially fast, $C(t) \to C_{\text{crit}}$ exponentially fast and hence $\left[4\pi + \frac{2(1-\genus)}{C(t)}\right] \to 0$ exponentially fast.
\end{remark}
\section{Convergence}
\label{sec-5}
\label{50880488-de51-46b5-9883-fc520d40172c}

\label{sec:convergence}

In this last section, let us briefly discuss the proof of Theorem \ref{thm:convergence}. The argument is very standard, following from bootstrapping the curvature bounds to higher derivative bounds. Here we will only outline the steps, indicating how the results here may be applied.

\begin{proof}
[Proof of Theorem \ref{thm:convergence}]
First, observe that by Theorems \ref{thm:sphere_bounds}, \ref{thm:plane_comparison} and, \ref{thm:ricci_hyperbolic_curvature_bound_negative} and, \ref{thm:ricci_hyperbolic_curvature_bound_general} and the fact that $\uniproj: \unicover{M} \to M$ is a local isoemetry, we have uniform upper bounds $\gausscurv_M \leq \gausscurv_0(t) + (1-\genus)$ with $\gausscurv_0(t)$ uniformly bounded and such that $\lim_{t\to\infty} \gausscurv_0(t) = 0$. Since the Gauss curvature evolves according to $\pd{t} \gausscurv = \laplace \gausscurv + \gausscurv(\gausscurv - (\genus - 1))$, by an ODE comparison we also have uniform lower bounds converging to $0$ as $t\to\infty$. Thus $\abs{\gausscurv(t)}$ is uniformly bounded for all $t>0$ and hence the solution exists for all time. 

$L^1$ convergence of the curvature to $1-\genus$ now follows easily. By Gauss-Bonnet, $\gausscurv \leq \gausscurv_0(t) + (1-\genus)$ and, the fact that $\abs{M} = 4\pi$,
\[
0 = \int_M \gausscurv - (1-\genus) d\mu \leq - \int_{\gausscurv \leq (1-\genus)} \abs{\gausscurv - (1-\genus)} d\mu + 4\pi \gausscurv_0(t),
\]
which rearranges to give $\int_{\gausscurv \leq (1-\genus)} \abs{\gausscurv - (1-\genus)} d\mu \leq 4\pi \gausscurv_0(t)$. Therefore  we get
\[
\int_M \abs{\gausscurv - (1-\genus)} d\mu \leq 8\pi \gausscurv_0(t)
\]
which converges to $0$ as $t\to\infty$.

Next we bound the higher derivatives of $\gausscurv$. By the bootstrapping argument described in \cite[Section 7]{MR1375255}, and from the uniform curvature bounds we obtain 
\[
\abs{\nabla^{(j)} \gausscurv}^2 \leq C_j ((1-\genus) + t^{-j})
\]
for constants $C_j>0$.

In the genus $\genus=0$ case, the lower bound on the isoperimetric constant affords very strong analytic control allowing us to apply Gagliardo-Nirenberg inequalities to deduce that $\gausscurv \to_{C^{\infty}} (1-\genus)$ uniformly as $t \to \infty$. See \cite{MR2729306}.

For higher genus surfaces, we don't have such control as noted in remark \ref{rem:isoconst_bdd}. Thus instead, for $\genus>0$ surfaces, with a little more work, another bootstrapping argument gives $\gausscurv \to_{C^{\infty}} (1-\genus)$ uniformly as $t \to \infty$ \cite[Section 7]{MR1375255}.

Finally, in the cases $\genus \ne 1$, we have $\gausscurv_0(t) = Ce^{-at}$ which gives for any non-zero $v\in TM$,
\[
\abs{\pd{t} \ln \metric(t) (v,v)} = 2\abs{\gausscurv - (1-\genus)} \leq C^{-at}
\]
which is integrable in $t$ on $[0,\infty)$. Smooth convergence of the metric now follows by the argument in \cite[Section 17]{MR664497}.

For the case $\genus = 0$ we only have $\gausscurv_0(t) = C/t$ which is not integrable. However, we may use the fact that $\abs{\nabla \gausscurv_0} \leq C/t^{3/2}$ which is integrable to again deduce smooth convergence \cite{MR2061425}.
\end{proof}

\begin{remark}
Note that we have control of the isoperimetric constant on $\unicover{M}$ and a curvature bound. We cannot however use these to obtain the simpler convergence proof using Gagliardo-Nirenberg inequalities since these rely on $L^1$ convergence of the curvature. But this is invalid for $\genus > 1$ since $\unicover{M}$ is not compact. Perhaps one might deduce $L^1_{\text{loc}}$ convergence, but note that the above $L^1$ convergence argument uses Gauss-Bonnet. In the $L^1_{\text{loc}}$ case, we would need to deal with boundary terms arising from Gauss-Bonnet and I don't know how to control these. This is perhaps related to transferring isoperimetric control from $\unicover{M}$ to $M$.
\end{remark}

\printbibliography

\end{document}